\documentclass[11pt]{amsart}
\newcommand{\NN}{\mathrm{I\!N\!}}
\makeatletter
\def\subsection{\@startsection{subsection}{3}
  \z@{.5\linespacing\@plus.7\linespacing}{.1\linespacing}
  {\normalfont\itshape}}
\makeatother

\usepackage{amssymb}
\usepackage{amsmath}
\usepackage{mathrsfs}
\usepackage{amsthm}
\usepackage{enumerate}
\newtheorem{theorem}{Theorem}[section]
\newtheorem{lemma}[theorem]{Lemma}
\newtheorem{corollary}[theorem]{Corollary}
\newtheorem{proposition}[theorem]{Proposition}
\theoremstyle{definition}
\newtheorem{problem}{Problem}
\newtheorem*{remark}{Remark}
\newtheorem*{remarks}{Remarks}
\numberwithin{equation}{section}
\mathchardef\hyphen="2D
\def\@tvsp{\mathchoice{{}\mkern-4.5mu}{{}\mkern-4.5mu}{{}\mkern-2.5mu}{}}
\def\ln{\left|\@tvsp\left|\@tvsp\left|}
\def\rn{\right|\@tvsp\right|\@tvsp\right|}
\makeatother
\begin{document}
\title{A coordinate free characterization of certain quasidiagonal operators}
\author{March T.~Boedihardjo}
\address{Department of Mathematics, Texas A\&M University, College Station, Texas 77843}
\email{march@math.tamu.edu}
\begin{abstract}
We obtain (i) a new, coordinate free, characterization of quasidiagonal operators with essential spectra contained in the unit circle by adapting the proof of a classical
result in the theory of Banach spaces, (ii) an affirmative answer to some questions of Hadwin, and (iii) an alternative proof of Hadwin's characterization of the SOT, WOT
and $*$-SOT closure of the unitary orbit of a given operator on a separable, infinite dimensional, complex Hilbert space.
\end{abstract}
\maketitle
\allowdisplaybreaks
\section{Introduction}\label{1}
In this paper, $\mathcal{H}$ is always a fixed separable, infinite dimensional, complex Hilbert space, and $\mathcal{B(H)}$ is the algebra of all operators (i.e., bounded
linear transformations) on $\mathcal{H}$. The ideal of compact operators in $\mathcal{B(H)}$ is denoted by $\mathcal{K(H)}$. If $\mathcal{H}_{1}$ and $\mathcal{H}_{2}$ are
Hilbert spaces, then $\mathcal{B}(\mathcal{H}_{1},\mathcal{H}_{2})$ denotes the set of all operators from $\mathcal{H}_{1}$ into $\mathcal{H}_{2}$.

Usually the connection between the theory of Banach spaces and the theory of operators on Hilbert space involve the study of spaces of operators as Banach spaces and vice
versa, Banach spaces as subspaces of $\mathcal{B(H)}$, or operators on Banach space as generalization of operators on Hilbert space. The purpose of this paper is to
illustate that new insights into operator theory can also be obtained from the theory of Banach spaces via the following Replacement Rule:

{\it Every Banach space is replaced by an operator, a complemented subspace of a Banach space is replaced by a reducing part of the corresponding operator, and an operator
between Banach spaces is replaced by an operator intertwining the corresponding operators.}

Using this Replacement Rule, we investigate the analogs in operator theory of (i) the problem of complementably universal Banach spaces and (ii) ultraproducts of Banach
spaces. The main consequences of this investigation are
\begin{enumerate}[(I)]
\item a coordinate free characterization of quasidiagonal operators with essential spectra contained in the unit circle (i.e., a characterization that does not require one
to find a decomposition of the space into finite dimensional subspaces or to find an appropriate sequence of projections converging strongly to the identity in order to
determine that a given operator with essential spectrum contained in the unit circle is quasidiagonal.);
\item the following result: Suppose that $T_{1},T_{2}\in\mathcal{B(H)}$. If $\lambda\geq 1$ and
\begin{equation}\label{11e}
T_{2}\in\{ST_{1}S^{-1}:S\in\mathcal{B(H)}\text{ with }\|S\|\|S^{-1}\|\leq\lambda\}^{-\|\,\|},
\end{equation}
then there exists a sequence $(S_{n})_{n\geq 1}$ of invertible operators on $\mathcal{H}$ with $\|S_{n}\|\|S_{n}^{-1}\|\leq\lambda$ such that $\displaystyle\lim_{n\to
\infty}\|T_{2}-S_{n}T_{1}S_{n}^{-1}\|=0$ and $T_{2}-S_{n}T_{1}S_{n}^{-1}\in\mathcal{K(H)}$.;
and
\item an alternative proof of Hadwin's characterization \cite{Hadwin2} of the SOT, WOT and $*$-SOT closure of the unitary orbit of a given operator on $\mathcal{H}$.
\end{enumerate}
The author later became aware that the proof of (II) answers affirmatively the following questions of Hadwin (see Question 1 and 9 in \cite{Hadwin3}.)

Hadwin defined two operators $T_{1},T_{2}\in\mathcal{B(H)}$ to be {\it approximately similar} if there exists $\lambda\geq 1$ satisfying (\ref{11e}) above. He asked
whether or not $T_{1},T_{2}\in\mathcal{B(H)}$ are approximately similar (if and) only if there exist $n\geq 1$ and $B_{1},B_{2},\ldots,B_{n}\in\mathcal{B(H)}$ such that
$B_{1}=T_{1}$, $B_{n}=T_{2}$ and for each $1\leq k<n$, $B_{k}$ is either similar or approximately unitarily equivalent to $B_{k+1}$. Hadwin also asked that in case this is
true, can we find one $n\geq 1$ that is valid for all $T_{1}$ and $T_{2}$ that are approximately similar.

He pointed out that if there exists such $n\geq 1$, it has to be at least 4 (see Example 1 in \cite{Hadwin3}) and that $n=4$ is valid under the assumption that the unital
$C^{*}$-algebra generated by $T_{1}$ and the unital $C^{*}$-algebra generated by $T_{2}$ are both disjoint from $\mathcal{K(H)}$ except for 0. From the proof of (II), we
obtain that $n=4$ is valid without this assumption. Indeed, we obtain that if $T_{1}$ and $T_{2}$ are approximately similar, then $T_{1}$ is approximately unitarily
equivalent to an operator $T_{1}'\in\mathcal{B(H)}$ that is similar to an operator $T_{2}'\in\mathcal{B(H)}$ that is approximately unitarily equivalent to $T_{2}$.

In Section 2, we consider the complementably universality problem for Banach spaces (see Problem \ref{21p} below), and we find an analogous problem in operator theory (see
Problem \ref{23p} below). Then by adapting (via the above Replacement Rule) the proof of a universality result of Johnson and Szankowski in \cite{Johnson}, we obtain a
partial solution to Problem \ref{23p} (see Theorem \ref{22}). This partial solution yields (I) above (see Corollary \ref{26}).

In Section 3, we consider ultraproducts of operators on Hilbert space, and use them, together with the Calkin representation \cite{Calkin} and Voiculescu's theorem
\cite{Voiculescu}, to obtain (II) and (III) above (see Theorem \ref{36} and Theorem \ref{37}, respectively). The connection of the results in this section to the theory of
Banach spaces is not clear at all at the first glance. But all the results were indeed inspired by ultraproducts of Banach spaces and a closely related concept finite
representability of Banach spaces. See the end of the section. The author later became aware that the technique used in this section is similar to that used in
\cite[Section 3]{Foias}.

We begin by introducing some terminology and notation that will be needed in what follows.

Subspaces are always assumed to be norm closed. Throughout this paper, we will systematically use the symbols $X,Y,Z$ for Banach spaces, $A,B,S,T$ for operators, $K$ for a
compact operator, $W$ for a unitary operator between Hilbert spaces, $P,Q$ for idempotents, and $I$ for the identity operator on a Banach space.
\subsection*{A. Operator theory}
Let $T_{1}\in\mathcal{B}(\mathcal{H}_{1})$ and $T_{2}\in\mathcal{B}(\mathcal{H}_{2})$. An operator $A\in\mathcal{B}(\mathcal{H}_{1},\mathcal{H}_{2})$ intertwines $T_{1}$
and $T_{2}$ if $AT_{1}=T_{2}A$.

The operators $T_{1}$ and $T_{2}$ are {\it compalent} \cite{Pearcy}, denoted by $T_{1}\stackrel{c}{\sim}T_{2
}$, if there exist a unitary operator $W\in\mathcal{B}(\mathcal{H}_{1},\mathcal{H}_{2})$ and a compact operator $K\in\mathcal{K}(\mathcal{H}_{2})$ such that
\[T_{2}=WT_{1}W^{*}+K;\]
$T_{1}$ and $T_{2}$ are {\it approximately unitarily equivalent} \cite{Voiculescu}, denoted by $T_{1}\simeq_{a}T_{2}$, if there exists a sequence $(W_{n})_{n\geq 1}$ of
unitary operators in $\mathcal{B}(\mathcal{H}_{1},\mathcal{H}_{2})$ such that $T_{2}-W_{n}T_{1}W_{n}^{*}\in\mathcal{K}(\mathcal{H}_{2})$ for all $n\geq 1$ and
\[\lim_{k\to\infty}\|T_{2}-W_{k}T_{1}W_{k}^{*}\|=0;\]
$T_{1}$ and $T_{2}$ are {\it unitarily equivalent}, denoted by $T_{1}\cong T_{2}$, if there exists a unitary operator $W\in\mathcal{B}(\mathcal{H}_{1}, \mathcal{H}_{2})$
such that $T_{2}=WT_{1}W^{*}$. The {\it unitary orbit} of an operator $T\in\mathcal{B(H)}$ is defined by
\[\mathcal{U}(T):=\{T_{0}\in\mathcal{B(H)}:T\cong T_{0}\}.\]

Let $T\in \mathcal{B(H)}$ and let $\mathcal{M}$ be a subspace of $\mathcal{H}$. An operator $T_{0}\in
\mathcal{B(M)}$ is a {\it restriction} of $T$ if $\mathcal{M}$ is invariant under $T$ and $T_{0}=T|_{\mathcal{M}}$; $T_{0}$ is a {\it reducing part} of $T$ if moreover
$\mathcal{M}$ is a {\it reducing subspace} for $T$, i.e., invariant under $T$ and $T^{*}$; $T_{0}$ is a {\it compression} of $T$ if $T_{0}=PT|_{\mathcal{M}}$ where $P$ is
the orthogonal projection from $\mathcal{H}$ onto $\mathcal{M}$.

The operator $T$ is {\it block diagonal} \cite{Halmos} if it is unitarily equivalent to a countably
infinite direct sum of operators each of which acts on a finite dimensional Hilbert space; $T$ is {\it quasidiagonal} \cite{Halmos} if it is the sum of a block diagonal
operator and a compact operator; $T$ is {\it subnormal} if it is the restriction of a normal operator; $T$ is {\it contractive} if $\|T\|\leq 1$. A contractive operator is
called a {\it contraction}.

Let $\pi$ be the quotient map from $\mathcal{B(H)}$ onto $\mathcal{B(H)}/\mathcal{K(H)}$. We write $\|T\|_{e}:=\|\pi(T)\|$ and $\sigma_{e}(T):=\sigma(\pi(T))$ for the
essential norm and the essential spectrum of $T$, respectively. A {\it representation} $\rho$ on a Hilbert space $\mathcal{H}_{0}$ of a unital $C^{*}$-algebra
$\mathcal{A}$ is a $*$-homomorphism from $\mathcal{A}$ into $\mathcal{B}(\mathcal{H}_{0})$. We say that $\rho$ is {\it unital} if $\rho(1)=I$. If $a\in\mathcal{A}$ then
the unital $C^{*}$-subalgebra of $\mathcal{A}$ generated by $a$ is denoted by $C^{*}(a)$.

As usual, the strong operator topology is denoted by SOT and the weak operator topology is denoted by WOT. A net $\{T_{\alpha}\}_{\alpha\in\Lambda}$ of operators in
$\mathcal{B(H)}$ converges in the {\it $*$-strong operator topology} if both $T_{\alpha}\to T$ and $T_{\alpha}^{*}\to T^{*}$ in SOT. This topology is denoted by
$*$-SOT.

The following known lemmas are stated here for the reader's convenience.
\begin{lemma}[\cite{Halmos}, page 903]\label{11}
Every contractive quasidiagonal opeator is the sum of a contractive block diagonal operator and a compact operator.
\end{lemma}
\begin{lemma}[\cite{Rosenblum}, Corollary 3.3]\label{12}
Let $\mathcal{H}_{1}$ and $\mathcal{H}_{2}$ be separable infinite dimensional complex Hilbert spaces. If $T_{1}\in\mathcal{B}(\mathcal{H}_{1})$ and $T_{2}\in \mathcal{B}(
\mathcal{H}_{2})$ have disjoint essential spectra and $A\in\mathcal{B}(\mathcal{H}_{1},\mathcal{H}_{2})$ intertwines $T_{1}$ and $T_{2}$, then $A$ is compact.
\end{lemma}
The following proposition may not have been noticed before and will be used in Theorem \ref{22}.
\begin{proposition}\label{13}
Let $A\in \mathcal{B(H)}$. If $I\in\{K\in \mathcal{K(H)}:KA=AK,\;K^{*}=K\}^{-SOT}$ then $A$ is block diagonal.
\end{proposition}
\begin{proof}
By assumption, there is a net $(K_{\alpha})$ of self-adjoint compact operators on $\mathcal{H}$ commuting with $A$ converging in SOT to $I$. Let $(\epsilon_{\alpha})$ be a
net of positive numbers converging to 0 (e.g., take $\epsilon_{\alpha}=\sum_{n=1}^{\infty}\frac{1}{2^{n}}\min(\|(K_{\alpha}-I)x_{n}\|,1)$ where $(x_{n})_{n\geq 1}$ is a
dense sequence in the unit ball of $\mathcal{H}$.)

Let $E$ be the spectral measure of $K_{\alpha}$ and let $P_{\alpha}=E(\mathbb{R}\backslash[-\epsilon_{\alpha},\epsilon_{\alpha}])$, the spectral projection of the given
set. Note that $P_{\alpha}$ commutes with $A$, since $P_{\alpha}$ is a sum of orthogonal projections onto $\text{ker}(K_{\alpha}-\lambda I)$ and $K_{\alpha}=K_{\alpha}^{*
}$ commutes with $A$. Since $P_{\alpha}x$ is the best approximation of $x$ by elements of $P_{\alpha}\mathcal{H}$,
\begin{eqnarray*}
\|x-P_{\alpha}x\|\leq\|x-P_{\alpha}K_{\alpha}x\|&\leq&\|x-K_{\alpha}x\|+\|(I-P_{\alpha})K_{\alpha}x\|\\&=&\|x-K_{\alpha}x\|+\|E([-\epsilon_{\alpha},\epsilon_{\alpha}])K_{n
}x\|\\&\leq&\|x-K_{\alpha}x\|+\epsilon_{\alpha}\|x\|\to 0,
\end{eqnarray*}
for all $x\in\mathcal{H}$. Thus, $P_{\alpha}\to I$ in SOT. But $P_{\alpha}$ commutes with $A$. Therefore,
\[I\in\{P\in\mathcal{B(H)}:P\text{ is a finite rank orthogonal projection and }PA=AP\}^{-SOT}.\]
Since this set of $P$ is uniformly bounded, there exists a sequence $(P_{n})_{n\geq 1}$ of finite rank orthogonal projection converging in SOT to $I\in\mathcal{B(H)}$ and
commuting with $A$. Let $Q_{n}$ be the orthogonal projection from $\mathcal{H}$ onto the closed subspace of $\mathcal{H}$ generated by $P_{1}\mathcal{H}\cup\ldots\cup
P_{n}\mathcal{H}$. Then $Q_{n}\to I$ in SOT, and $Q_{k}\leq Q_{k+1}$ and $Q_{k}$ commutes with $A$ for all $k\geq 1$. Hence, $A$ is block diagonal.
\end{proof}
\begin{remark}
One consequence of this result is that a reducing part of a block diagonal operator is also block diagonal, which is perhaps a known fact.
\end{remark}
\subsection*{B. Theory of Banach spaces}
Two Banach spaces $X$ and $Y$ are {\it isomorphic} if there is an {\it isomorphism} from $X$ onto $Y$, i.e., a linear homeomorphism from $X$ onto $Y$. A subspace $Z$ of
$X$ is said to be {\it complemented} if there is an idempotent  from $X$ onto $Z$. We say that the Banach space $X$ has the {\it bounded compact approximation property}
({\it BCAP}) if there is a uniformly bounded net of compact operators on $X$ converging in SOT to $I$.

General results in operator theory can be found in \cite{Douglas} and \cite{Radjavi}. For an introduction to compalence of operators, the reader is referred to
\cite{Pearcy}. General results in the theory of Banach spaces can be found in \cite{LinTza} and \cite{Diestel}.
\section{Universal Banach spaces and universal operators}\label{2}
The motivation for the work in this section derives from the following classical problem in the theory of Banach spaces.
\begin{problem}\label{21p}
For a given class $\mathcal{C}$ of separable Banach spaces, does there exist a separable Banach space $X$ such that every Banach space in $\mathcal{C}$ is isomorphic to a
complemented subspace of $X$?
\end{problem}
The Replacement Rule introduced at the beginning of this paper suggests that an analog of this problem in operator theory could be
\begin{problem}\label{22p}
For a given class $\mathcal{C}$ of uniformly bounded operators in $\mathcal{B(H)}$ (i.e., $\sup\{\|T\|:T\in\mathcal{C}\}<\infty$), does there exist an operator $T\in
\mathcal{B(H)}$ such that every operator in $\mathcal{C}$ is unitarily equivalent to a reducing part of $T$?
\end{problem}
The answer is trivially yes, if $\mathcal{C}$ is countable, by considering the direct sum of all operators in $\mathcal{C}$. On the other hand, the answer is no even
for the class $\{\alpha I:\alpha\in [0,1]\}$ (which is uncountable). To see this, suppose that $T\in\mathcal{B(H)}$ and that for every $\alpha\in [0,1]$, there is an
infinite dimensional reducing subspace $\mathcal{H}_{\alpha}$ for $T$ such that $T|_{\mathcal{H}_{\alpha}}=\alpha I$. Letting $P_{\alpha}$ be the orthogonal projection
onto $\mathcal{H}_{\alpha}$, we have
\[TP_{\alpha}P_{\beta}=T|_{\mathcal{H}_{\alpha}}P_{\alpha}P_{\beta}=\alpha P_{\alpha}P_{\beta}.\]
Similarly, we have $P_{\alpha}TP_{\beta}=\beta P_{\alpha}P_{\beta}$. But since $T$ commutes with $P_{\alpha}$, it follows that $P_{\alpha}P_{\beta}=0$ if $\alpha\neq
\beta$. Therefore, $\mathcal{H}_{\alpha}\perp\mathcal{H}_{\beta}$ if $\alpha\neq\beta$. Since there are uncountably many $\alpha$, this implies that $\mathcal{H}$ is not
separable, which is a contradiction.

In general, the answer to Problem \ref{22p} is no. Thus, we might obtain a more interesting problem if we replace unitary equivalence with a weaker equivalence, namely
with compalence.
\begin{problem}\label{23p}
For a given class $\mathcal{C}$ of uniformly bounded operators in $\mathcal{B(H)}$, does there exist an operator $T\in\mathcal{B(H)}$ such that every operator in
$\mathcal{C}$ is compalent to a reducing part of $T$?
\end{problem}
Is there, for instance, an operator $T\in\mathcal{B(H)}$ for which every multiple of $I\in\mathcal{B(H)}$ by a scalar in $[0,1]$ is compalent to a reducing part of $T$?

The answer is yes. An example is given by a diagonal operator $T$ with diagonal entries $\alpha_{1},\alpha_{2},\ldots$ in $\mathbb{R}$ satisfying $\{\alpha_{n}:n\geq 1\}^{
-}=[0,1]$. Then for each $\alpha\in [0,1]$, there is a subsequence $(\alpha_{n_{k}})_{k\geq 1}$ converging to $\alpha$. Hence, $\alpha I$ is a compact perturbation of a
diagonal operator $B$ with diagonal entries $\alpha_{n_{1}},\alpha_{n_{2}},\ldots$. But $B$ is (unitarily equivalent to) a reducing part of $T$, and therefore, $\alpha I$
is compalent to a reducing part of $T$. Is there an operator $T\in\mathcal{B(H)}$ for which every multiple of the unilateral shift (of multiplicity 1) by a scalar in
$[0,1]$ is compalent to a reducing part of $T$? (See Corollary \ref{27} below.) What about the bilateral shift?

For the class $(CQD)$ of contractive quasidiagonal operators, we give below an affirmative answer to Problem \ref{23p}. This yields, in particular, an affirmative answer
to the preceding question about the bilateral shift, since every normal operator is quasidiagonal \cite[page 903]{Halmos}.
\begin{theorem}\label{21}
There is a contractive quasidiagonal operator $T\in \mathcal{B(H)}$ such that every contractive quasidiagonal operator is compalent to a reducing part of $T$.
\end{theorem}
\begin{proof}
For each $n\geq 1$, let $(T_{i,n})_{i\geq 1}$ be a dense sequence in the unit ball of $\mathcal{B}(\mathcal{H}_{n})$, where $\mathcal{H}_{n}$ is an $n$-dimensional Hilbert
space. Then set
\[T=\bigoplus_{i,n\geq 1}T_{i,n}.\]
If $A$ is a contractive quasidiagonal operator then by Lemma \ref{11}, $A$ is a compact perturbation of a contractive block diagonal operator $B$. It can be easily
checked that $B$ is compalent to a reducing part of $T$. Therefore, $A$ is compalent to a reducing part of $T$.
\end{proof}
\begin{remark}
The construction of $T$ in Theorem \ref{21} is the same as the construction of the universal operator in \cite[Corollary 4.2]{Herrero}. But the notion
of universality in \cite[Corollary 4.2]{Herrero}, when restricted to $(CQD)$, is weaker than that in Theorem \ref{21}.
\end{remark}
The main result of this section is that under an additional assumption, the quasidiagonal operators actually characterize the existence of a universal operator.
\begin{theorem}\label{22}
Suppose that $\mathcal{C}$ is a collection of uniformly bounded operators in $\mathcal{B(H)}$ such that $\sigma_{e}(S_{1})\cap\sigma_{e}(S_{2})=\emptyset$ for all $S_{1},
S_{2}\in\mathcal{C},\,S_{1}\neq S_{2}$. Then the following statements are equivalent.
\begin{enumerate}[(i)]
\item There exists an operator $T\in\mathcal{B(H)}$ such that every operator in $\mathcal{C}$ is compalent to a reducing part of $T$.
\item Every operator in $\mathcal{C}$ outside a countable subset is quasidiagonal.
\end{enumerate}
\end{theorem}
\begin{remarks}
In Theorem \ref{22}, we can replace uniform boundedness of $\mathcal{C}$ by essential uniform boundedness, i.e., $\sup\{\|T\|_{e}:T\in\mathcal{C}\}<\infty$. A slightly
stronger statement of Theorem \ref{21} is also true: There is a contractive quasidiagonal operator $T\in \mathcal{B(H)}$ such that every quasidiagonal operator with
essential norm at most 1 is compalent to a reducing part of $T$.
\end{remarks}
The universality result mentioned in the introduction which led to Theorem \ref{22} is the following.
\begin{theorem}[\cite{Johnson}, Section II]\label{23}
There is no separable Banach space $X$ such that every separable Banach space is isomorphic to a complemented subspace of $X$.
\end{theorem}
The proof in \cite{Johnson} uses the following fact about Banach spaces: There are separable Banach spaces $E_{p}$ where $1<p<\infty$ such that (a) $E_{p}$ fails the BCAP
for every $1<p<\infty$ and (b) if $q<r$ then every operator from a subspace of $E_{r}$ into $E_{q}$ is compact. Then the result follows from the following lemma. (This
lemma is in fact not stated in \cite{Johnson} but is extracted from the original proof of Theorem \ref{23} in \cite{Johnson}.)
\begin{lemma}\label{24}
Suppose that $E_{p}$ is a separable Banach space where $1<p<\infty$ such that
\begin{enumerate}[(a)]
\item $E_{p}$ fails the BCAP for each $1<p<\infty$ and
\item if $q<r$ then every operator from $E_{r}$ to $E_{q}$ is compact.
\end{enumerate}
Then there is no separable Banach space $X$ such that for every $1<p<\infty$, $E_{p}$ is isomorphic to a complemented subspace of $X$.
\end{lemma}
\begin{proof}
Suppose, on the contrary, that there is a separable Banach space $X$ such that for every $1<p<\infty$, $E_{p}$ is isomorphic to a complemented subspace $Y_{p}$ of $X$.
Letting $Q_{p}$ be an idempotent from $X$ onto $Y_{p}$, we have that there exist $M\in\NN\,$ and an uncountable set $\mathscr{A}\subset (1,\infty)$ so that $\|Q_{p}\|
\leq M$ for each $p\in\mathscr{A}$.

For each $p\in\mathscr{A}$, since $E_{p}$ fails the BCAP, $Y_{p}$ fails the BCAP so $I\notin\{K:Y_{p}\to Y_{p}:K\text{ is compact and }\|K\|\leq M^{2}\}^{-SOT}$. Thus,
there is an SOT-open neighborhood of $I$ on $Y_{p}$ that is disjoint from $\{K:Y_{p}\to Y_{p}:K\text{ is compact and }\|K\|\leq M^{2}\}$. By definition of SOT, this means
that there exist a finite set $(y_{i}^{p})_{i=1}^{n(p)}$ of unit vectors in $Y_{p}$ and $\epsilon_{p}>0$ so that there is no compact operator $K$ on $Y_{p}$ for which
$\|K\|\leq M^{2}$ and $\|y_{i}^{p}-Ky_{i}^{p}\|<\epsilon_{p}$ for $1\leq i\leq n(p)$. Choose an uncountable subset $\mathscr{B}$ of $\mathscr{A}$ so that $n(p)$ is
constant (say $=n$) on $\mathscr{B}$ and $\displaystyle\inf_{p\in\mathscr{B}}\epsilon_{p}=\epsilon>0$.

Since $\mathscr{B}$ is uncountable and $X$ is separable, there exist $q<r$ in $\mathscr{B}$ so that $\|y_{i}^{q}-y_{i}^{r}\|<(M+M^{2})^{-1}\epsilon$ for
$1\leq i\leq n$. Let $K_{0}:Y_{r}\to Y_{r}$ be the restriction of $Q_{r}Q_{q}$ to $Y_{r}$. Then the following properties of $K_{0}$ are valid.
\begin{enumerate}[(i)]
\item $K_{0}$ is compact, since $Q_{q}|_{Y_{r}}:Y_{r}\rightarrow Y_{q}$ is compact by assumption (b);
\item $\|K_{0}\|\leq M^{2}$; and
\item $\|y_{i}^{r}-K_{0}y_{i}^{r}\|<\epsilon_{r}$ for $1\leq i\leq r$. Indeed,
\begin{eqnarray*}
\|y_{i}^{r}-K_{0}y_{i}^{r}\|&=&\|y_{i}^{r}-Q_{r}Q_{q}y_{i}^{r}\|\\&=&\|Q_{r}y_{i}^{r}-Q_{r}Q_{q}y_{i}^{r}\|\\&\leq&M\|(I-Q_{q})y_{i}^{r}\|\\&=&M\|(I-Q_{q})(y_{i
}^{r}-y_{i}^{q})\|\leq M(1+M)\|y_{i}^{r}-y_{i}^{q}\|<\epsilon\leq\epsilon_{r}.
\end{eqnarray*}
\end{enumerate}
These properties of $K_{0}$ contradict the choice of $(y_{i}^{r})_{i=1}^{n}$ and the proof is complete.
\end{proof}
The preceding lemma can be adapted to the context of operator theory via the Replacement Rule.
\begin{lemma}\label{25}
Suppose that $(S_{\alpha})_{\alpha\in\Lambda}$ is an uncountable indexed collection of non-quasidiagonal operators such that
\begin{enumerate}[(a)]
\item $S_{\alpha}$ is not quasidiagonal for each $\alpha\in\Lambda$ and
\item if $\beta\neq\gamma$ then every operator intertwining $S_{\beta}$ and $S_{\gamma}$ is compact.
\end{enumerate}
Then there is no operator $T\in\mathcal{B(H)}$ such that for every $\alpha\in\Lambda$, $S_{\alpha}$ is compalent to a reducing part of $T$.
\end{lemma}
\begin{proof}
Suppose, on the contrary, that there is an operator $T\in\mathcal{B(H)}$ such that for every $\alpha\in\Lambda$, $S_{\alpha}$ is compalent to a reducing part $T_{\alpha}:=
T|_{\mathcal{H}_{\alpha}}$ of $T$ where $\mathcal{H}_{\alpha}$ is a reducing subspace for $T$. Let $P_{\alpha}$ be the orthogonal projection from $\mathcal{H}$ onto
$\mathcal{H}_{\alpha}$.

For each $\alpha\in\Lambda$, since $S_{\alpha}$ is not quasidiagonal, $T_{\alpha}$ is not block diagonal so by Proposition \ref{13}, $I\notin\{K\in\mathcal{K(H)}:KT_{
\alpha}=T_{\alpha}K\text{ and }K^{*}=K\}^{-SOT}$. Thus, there is an SOT-open neighborhood of $I\in\mathcal{B(H)}$ that is disjoint from $\{K\in\mathcal{K(H)}:KT_{\alpha}=
T_{\alpha}K\text{ and }K^{*}=K\}$. By definition of SOT, this means that there exist a finite set $(x_{i}^{\alpha})_{i=1}^{n(\alpha)}$ of unit vectors in $\mathcal{H}_{
\alpha}$ and $\epsilon_{\alpha}>0$ so that there is no self-adjoint compact operator $K$ on $\mathcal{H}_{\alpha}$ commuting with $T_{\alpha}$ for which
$\|x_{i}^{\alpha}-Kx_{i}^{\alpha}\|<\epsilon_{\alpha}$ for $1\leq i\leq n(\alpha)$. Choose an uncountable subset $\mathscr{B}$ of $\Lambda$ so that $n(\alpha)$ is constant
(say $=n$) on $\mathscr{B}$ and $\displaystyle\inf_{\alpha\in \mathscr{B}}\epsilon_{\alpha}=\epsilon>0$.

Since $\mathscr{B}$ is uncountable and $\mathcal{H}$ is separable, there exist $\beta\neq\gamma$ in $\mathscr{B}$ so that $\|x_{i}^{\beta}-x_{i}^{\gamma}\|<\epsilon$ for
$1\leq i\leq n$. Let $K_{0}\in\mathcal{B}(\mathcal{H}_{\gamma})$ be the restriction of $P_{\gamma}P_{\beta}$ to $\mathcal{H}_{\gamma}$. Then the following properties of
$K_{0}$ are valid.
\begin{enumerate}[(i)]
\item $K_{0}$ is self-adjoint.
\item $K_{0}$ is compact. Indeed, $P_{\beta}|_{\mathcal{H}_{\gamma}}$ intertwines $T_{\gamma}$ and $T_{\beta}$ and thus is compact by assumption.
\item $K_{0}$ commutes with $T_{\gamma}$. Indeed, since $\mathcal{H}_{\beta}$ and $\mathcal{H}_{\gamma}$ are reducing subspaces for $T$, $P_{\beta}$ and $P_{\gamma}$
commute with $T$. Thus, $P_{\gamma}P_{\beta}T=TP_{\gamma}P_{\beta}$ and so $P_{\gamma}P_{\beta}|_{\mathcal{H}_{\gamma}}T|_{\mathcal{H}_{\gamma}}=T|_{\mathcal{H}_{\gamma}}
P_{\gamma}P_{\beta}|_{\mathcal{H}_{\gamma}}$. Hence, $K_{0}T_{\gamma}=T_{\gamma}K_{0}$.
\item $\|x_{i}^{\gamma}-K_{0}x_{i}^{\gamma}\|<\epsilon_{\gamma}$ for $1\leq i\leq n$. Indeed,
\begin{eqnarray*}
\|x_{i}^{\gamma}-K_{0}x_{i}^{\gamma}\|&=&\|x_{i}^{\gamma}-P_{\gamma}P_{\beta}x_{i}^{\gamma}\|\\&=&\|P_{\gamma}x_{i}^{\gamma}-P_{\gamma}P_{\beta}x_{i}^{\gamma}\|\\&\leq&
\|(I-P_{\beta})x_{i}^{\gamma}\|=\|(I-P_{\beta})(x_{i}^{\gamma}-x_{i}^{\beta})\|\leq\|x_{i}^{\gamma}-x_{i}^{\beta}\|<\epsilon\leq\epsilon_{\gamma}.
\end{eqnarray*}
\end{enumerate}
These properties of $K_{0}$ contradict the choice of $(x_{i}^{\gamma})_{i=1}^{n}$ and the proof is complete.
\end{proof}
The replacements needed to transform Lemma \ref{24} to Lemma \ref{25} are
\begin{center}
\begin{tabular}{ll}
Lemma \ref{24}&Lemma \ref{25}\\\\
$E_{p}$&$S_{\alpha}$\\\\
$X$&$T$\\\\
$Y_{p}$&$T_{\alpha}$\\\\
$Q_{p}$&$P_{\alpha}$\\\\
$(y_{i}^{p})_{i=1}^{n(p)}$&$(x_{i}^{\alpha})_{i=1}^{n(\alpha)}$\\\\
$K:Y_{p}\to Y_{p}$ compact&$K\in\mathcal{K(H)}$ commuting with $T_{\alpha}$,\\&i.e, $K\in\mathcal{K(H)}$ intertwining $T_{\alpha}$ and $T_{\alpha}$\\\\
$\mathscr{B}$&$\mathscr{B}$
\end{tabular}
\end{center}
Also, $E_{p}$ failing the BCAP is replaced by $S_{\alpha}$ not being quasidiagonal. This is because $E_{p}$ failing the BCAP means that for every $M>0$,
\[I\notin\{K:E_{p}\to E_{p}:K\text{ is compact and }\|K\|\leq M\}^{-SOT},\]
which, according to the Replacement Rule, could be replaced by
\[I\notin\{K\in\mathcal{K(H)}:KS_{\alpha}=S_{\alpha}K\text{ and }\|K\|\leq M\}^{-SOT}.\]
But in the context of operator theory, it is natural to add the condition that $K^{*}=K$ so $E_{p}$ failing the BCAP could be replaced by
\[I\notin\{K\in\mathcal{K(H)}:KS_{\alpha}=S_{\alpha}K,\;\|K\|\leq M\text{ and }K^{*}=K\}^{-SOT},\]
which is equivalent to $S_{\alpha}$ not being block diagonal by Proposition \ref{13}. But BCAP preserves isomorphism whereas block diagonality does not preserve
compalence. So $E_{p}$ failing the BCAP should be replaced by $S_{\alpha}$ not being quasidiagonal.
\begin{proof}[Proof of Theorem \ref{22}]
That (ii)$\Rightarrow$(i) follows easily from Theorem \ref{21}. To prove Not (ii)$\Rightarrow$Not (i), suppose that (ii) is not true, i.e., there are uncountably many
non-quasidiagonal operators in $\mathcal{C}$. By assumption and Lemma \ref{12}, every operator intertwining two different operators in $\mathcal{C}$ is compact. Thus, by
Lemma \ref{25}, (i) is not true.
\end{proof}
We conclude this section with two corollaries of Theorem \ref{22}. The first one is a direct consequence of Theorem \ref{22} while the second one easily follows from the
first one since a Fredholm operator that is quasidiagonal must have index 0.
\begin{corollary}\label{26}
Let $T_{0}\in \mathcal{B(H)}$ with $\sigma_{e}(T_{0})\subset\{z\in\mathbb{C}:|z|=1\}$. Then $T_{0}$ is quasidiagonal if and only if there is an operator $T\in \mathcal{B(
H)}$ such that for every $\alpha\in [0,1]$, $\alpha T_{0}$ is compalent to a reducing part of $T$.
\end{corollary}
\begin{corollary}\label{27}
Let $U$ be the unilateral shift. There is no operator $T\in \mathcal{B(H)}$ such that for every $\alpha\in [0,1]$, $\alpha U$ is compalent to a reducing part of $T$. In
particular, there is no operator $T\in \mathcal{B(H)}$ such that every contraction in $\mathcal{B(H)}$ is compalent to a reducing part of $T$.
\end{corollary}
\section{Ultraproducts of operators}\label{3}
We begin by recalling from \cite{Reid} a slight reformulation of the construction of the Calkin representation in the language of ultraproducts.

Let $\mathscr{U}$ be a free ultrafilter on $\NN\,$. If $(a_{n})_{n\geq 1}$ is a bounded sequence in $\mathbb{C}$, then its ultralimit through $\mathscr{U}$ is denoted by
$\displaystyle\lim_{n,\mathscr{U}}a_{n}$. Consider the Banach space
\[\mathcal{H}^{\mathscr{U}}:=l_{\infty}(\mathcal{H})/\left\{(x_{n})_{n\geq 1}\in l_{\infty}(\mathcal{H}):\lim_{n,\mathscr{U}}\|x_{n}\|=0\right\}.\]
If $(x_{n})_{n\geq 1}\in l_{\infty}(\mathcal{H})$ then its image in $\mathcal{H}^{\mathscr{U}}$ is denoted by $(x_{n})_{\mathscr{U}}$, and it can be easily checked that
\[\|(x_{n})_{\mathscr{U}}\|=\lim_{n,\mathscr{U}}\|x_{n}\|.\]
Moreover, $\mathcal{H}^{\mathscr{U}}$ is, in fact, a Hilbert space with inner product
\[\langle(x_{n})_{\mathscr{U}},(y_{n})_{\mathscr{U}}\rangle=\lim_{n,\mathscr{U}}\langle x_{n},y_{n}\rangle.\]
But $\mathcal{H}^{\mathscr{U}}$ is nonseparable (see, e.g., \cite[Proposition 8.5]{Diestel}).

If $(T_{n})_{n\geq 1}$ is a bounded sequence in $\mathcal{B(H)}$, then its {\it ultraproduct} $(T_{1},T_{2},\ldots)_{\mathscr{U}}\in\mathcal{B}(\mathcal{H}^{\mathscr{U}})$
is defined by $(x_{n})_{\mathscr{U}}\mapsto (T_{n}x_{n})_{\mathscr{U}}$. If $T\in\mathcal{B(H)}$ then its {\it ultrapower} $T^{\mathscr{U}}\in\mathcal{B}(\mathcal{H}^{
\mathscr{U}})$ is defined by $(x_{n})_{\mathscr{U}}\mapsto (Tx_{n})_{\mathscr{U}}$. It is easy to see that
\[(T_{1},T_{2},\ldots)_{\mathscr{U}}^{*}=(T_{1}^{*},T_{2}^{*},\ldots)_{\mathscr{U}},\]
and in particular, $(T^{\mathscr{U}})^{*}=(T^{*})^{\mathscr{U}}$.

We pause here for a while to show that the strong limit of a sequence of normal operators on $\mathcal{H}$ is subnormal, using the ultraproduct construction. A stronger
result was proved in \cite[Theorem 3.3]{Bishop} and also in \cite{Conway} where the strong limit of a net of normal operators on $\mathcal{H}$ was shown to be subnormal.
Suppose that $(T_{n})_{n\geq 1}$ is a sequence of normal operators on $\mathcal{H}$ converging in SOT to $T\in\mathcal{B(H)}$. The uniform boundedness principle gives
$\displaystyle\sup_{n\geq 1}\|T_{n}\|<\infty$. Hence, the ultraproduct $(T_{1},T_{2},\ldots)_{\mathscr{U}}$ is well defined and is normal. Moreover, $\{(x)_{\mathscr{U}}:x
\in\mathcal{H}\}$ is invariant under this operator, and $T\cong(T_{1},T_{2},\ldots)_{\mathscr{U}}|_{\{(x)_{\mathscr{U}}:x\in\mathcal{H}\}}$. Therefore, $T$ is subnormal.

Consider the subspace
\[\widehat{\mathcal{H}}:=\left\{(x_{n})_{\mathscr{U}}\in \mathcal{H}^{\mathscr{U}}:w\hyphen\lim_{n,\mathscr{U}}x_{n}=0\right\}.\]
Here $\displaystyle w\hyphen\lim_{n,\mathscr{U}}x_{n}$ is the weak limit of $(x_{n})_{n\geq 1}$ through $\mathscr{U}$, i.e., the unique element $x\in\mathcal{H}$ such that
\[\langle x,y\rangle=\lim_{n,\mathscr{U}}\langle x_{n},y\rangle,\quad y\in\mathcal{H}.\]
Note that $\{(x)_{\mathscr{U}}:x\in\mathcal{H}\}^{\perp}=\widehat{\mathcal{H}}$, and thus,
\[\widehat{\mathcal{H}}^{\perp}=\{(x)_{\mathscr{U}}:x\in\mathcal{H}\}.\]
The orthogonal projection from $\mathcal{H}^{\mathscr{U}}$ onto $\widehat{\mathcal{H}}^{\perp}$ is given by $\displaystyle (x_{n})_{\mathscr{U}}\mapsto(w\hyphen\lim_{k,
\mathscr{U}}x_{k})_{\mathscr{U}}$. We shall identify the space $\widehat{\mathcal{H}}^{\perp}$ with $\mathcal{H}$ in the natural way. So we have $\mathcal{H}^{\mathscr{U}}=\mathcal{H}\oplus\mathcal{\widehat{H}}$.

For $T\in\mathcal{B(H)}$, $\mathcal{\widehat{H}}$ is a reducing subspace for $T^{\mathscr{U}}$ and define $\widehat{T}\in\mathcal{B}(\mathcal{\widehat{H}})$ by
\[\widehat{T}:=T^{\mathscr{U}}|_{\mathcal{\widehat{H}}}.\]
Thus, we have
\[T^{\mathscr{U}}=T\oplus\widehat{T}\]
with respect to the decomposition $\mathcal{H}^{\mathscr{U}}=\mathcal{H}\oplus\widehat{\mathcal{H}}$.

Note that $\widehat{K}=0$ for $K\in\mathcal{K(H)}$. (The proof uses the fact that every sequence in a compact metric space converges to an element through $\mathscr{U}$.) The map $f:\mathcal{B(H)}/\mathcal{K(H)}\to\mathcal{B}(\widehat{\mathcal{H}})$ defined by $\pi(T)\mapsto\widehat{T}$
is called the {\it Calkin representation}.
\begin{theorem}[\cite{Calkin}, Theorem 5.5]\label{31}
The map $f$ is an isometric $*$-isomorphism into $\mathcal{B}(\mathcal{\widehat{H}})$.
\end{theorem}

The following lemma will be useful throughout this section.
\begin{lemma}\label{32}
Let $T_{1},T_{3}\in \mathcal{B(H)}$ and let $T_{2},T_{4}$ be operators on a (not necessarily separable) Hilbert space $\widetilde{\mathcal{H}}$. If $T_{1}\oplus T_{2}\cong
T_{3}\oplus T_{4}$ then there is a separable reducing subspace $\mathcal{M}$ for both $T_{2}$ and $T_{4}$ such that
\[T_{1}\oplus(T_{2}|_{\mathcal{M}})\cong T_{3}\oplus(T_{4}|_{\mathcal{M}}).\]
\end{lemma}
\begin{proof}
Let $W$ be a unitary operator on $\mathcal{H}\oplus\widetilde{\mathcal{H}}$ such that
\[W(T_{1}\oplus T_{2})=(T_{3}\oplus T_{4})W.\]
Let\[\mathcal{N}=\{Sy:y\in\mathcal{H}\oplus\{0\}\text{ and }S\in C^{*}(T_{1}\oplus T_{2},T_{3}\oplus T_{4},W)\}^{-\|\,\|},\]
where $C^{*}(T_{1}\oplus T_{2},T_{3}\oplus T_{4},W)$ is the unital $C^{*}$-subalgebra of $\mathcal{B}(\mathcal{H}\oplus\widetilde{\mathcal{H}})$ generated by $T_{1}\oplus
T_{2}$, $T_{3}\oplus T_{4}$ and $W$. Then $\mathcal{N}$ is a separable reducing subspace for $T_{1}\oplus T_{2}$, $T_{3}\oplus T_{4}$ and $W$. Since $\mathcal{H}\oplus\{0
\}\subset\mathcal{N}$, there exists a subspace $\mathcal{M}\subset\widetilde{\mathcal{H}}$ such that $\mathcal{N}=\mathcal{H}\oplus\mathcal{M}$, and thus $\mathcal{M}$ is
a separable reducing subspace for $T_{2}$ and $T_{4}$. Moreover, since $\mathcal{N}$ reduces $W$, $W|_{\mathcal{N}}$ is a unitary operator on $\mathcal{N}$ and satisfies
\[(W|_{\mathcal{N}})(T_{1}\oplus(T_{2}|_{\mathcal{M}}))=(T_{3}\oplus(T_{4}|_{\mathcal{M}}))(W|_{\mathcal{N}}).\]
Therefore,\[T_{1}\oplus(T_{2}|_{\mathcal{M}})\cong T_{3}\oplus(T_{4}|_{\mathcal{M}}).\]
\end{proof}
The Calkin representation yields an alternative proof of the following known result (see, e.g., \cite[Theorem 2.29]{Pearcy}):

{\it If $T,K\in\mathcal{K(H)}$ and $T\simeq_{a}K$ then $T\oplus 0_{\mathcal{H}}\cong K\oplus 0_{\mathcal{H}}$, where $0_{\mathcal{H}}$ is the zero operator on
$\mathcal{H}$.}

Since $T\simeq_{a}K$, there exists a sequence $(W_{n})_{n\geq 1}$ of unitary operators on $\mathcal{H}$ such that $\displaystyle\lim_{n\to\infty}\|T-W_{n}KW_{n}^{*}\|=0$.
Thus,
\[T^{\mathscr{U}}=(W_{1}KW_{1}^{*},W_{2}KW_{2}^{*},\ldots)_{\mathscr{U}}=(W_{1},W_{2},\ldots)_{\mathscr{U}}K^{\mathscr{U}}(W_{1},W_{2},\ldots)_{\mathscr{U}}^{*},\]
and so $T^{\mathscr{U}}\cong K^{\mathscr{U}}$. Since $T,K\in\mathcal{K(H)}$, this implies that
\[T\oplus 0_{\widehat{\mathcal{H}}}=T\oplus\widehat{T}=T^{\mathscr{U}}\cong K^{\mathscr{U}}=K\oplus\widehat{K}=K\oplus 0_{\widehat{\mathcal{H}}}.\]
By Lemma \ref{32}, $T\oplus 0_{\mathcal{H}}\cong K\oplus 0_{\mathcal{H}}$.

Let us recall a result of Voiculescu.
\begin{theorem}[\cite{Voiculescu}, Theorem 1.3]\label{33}
Let $T\in\mathcal{B(H)}$ and let $\rho$ be a unital representation of $C^{*}(\pi(T))$ on a separable Hilbert space $\mathcal{H}_{\rho}$. Then $T\simeq_{a}T\oplus\rho(
\pi(T))$.
\end{theorem}
If $T\in\mathcal{B(H)}$ and $\mathcal{M}$ is a separable reducing subspace for $\widehat{T}$, then $\pi(S)\to\widehat{S}|_{M}$ defines a unital representation of $C^{*}(
\pi(T))$ on $\mathcal{M}$. Applying Theorem \ref{33} to this representation, we obtain
\begin{theorem}\label{34}
Let $T\in \mathcal{B(H)}$ and let $\mathcal{M}$ be a separable reducing subspace for $\widehat{T}$. Then $T\simeq_{a}T\oplus(\widehat{T}|_{\mathcal{M}})$.
\end{theorem}
\begin{theorem}\label{35}
If $T_{1},T_{2}\in\mathcal{B(H)}$ then $T_{1}\simeq_{a}T_{2}$ if and only if $T_{1}^{\mathscr{U}}\cong T_{2}^{\mathscr{U}}$.
\end{theorem}
\begin{proof}
If $T_{1}\simeq_{a}T_{2}$ then from a similar argument as in the discussion preceding Theorem \ref{33}, we have $T_{1}^{\mathscr{U}}\cong T_{2}^{\mathscr{U}}$. Conversely, suppose that $T_{1}^{
\mathscr{U}}\cong T_{2}^{\mathscr{U}}$. Then $T_{1}\oplus\widehat{T}_{1}\cong T_{2}\oplus\widehat{T}_{2}$, and thus by Lemma \ref{32}, there exists a separable reducing
subspace $\mathcal{M}$ for both $\widehat{T}_{1}$ and $\widehat{T}_{2}$ such that
\[T_{1}\oplus(\widehat{T}_{1}|_{\mathcal{M}})\cong T_{2}\oplus(\widehat{T}_{2}|_{\mathcal{M}}).\]
Thus, by Theorem \ref{34}, we obtain $T_{1}\simeq_{a}T_{2}$.
\end{proof}
Although we will not make use of Theorem \ref{35}, the proofs of the results below resemble the proof of this theorem.

Let $\mathcal{H}_{1}$ and $\mathcal{H}_{2}$ be Hilbert spaces and let $\lambda\geq 1$. Then two operators $T_{1}\in\mathcal{B}(\mathcal{H}_{1})$ and $T_{2}\in\mathcal{B}(
\mathcal{H}_{2})$ are {\it $\lambda$-similar} if there is an invertible operator $S\in\mathcal{B}(\mathcal{H}_{1},\mathcal{H}_{2})$ such that $T_{2}=ST_{1}S^{-1}$ and
$\|S\|\|S^{-1}\|\leq\lambda$.
\begin{theorem}\label{36}
Suppose that $T_{1},T_{2}\in\mathcal{B(H)}$. If $\lambda\geq 1$ and
\[T_{2}\in\{ST_{1}S^{-1}:S\in\mathcal{B(H)}\text{ with }\|S\|\|S^{-1}\|\leq\lambda\}^{-\|\,\|},\]
then there exists a sequence $(S_{n})_{n\geq 1}$ of invertible operators on $\mathcal{H}$ with $\|S_{n}\|\|S_{n}^{-1}\|\leq\lambda$ such that $\displaystyle\lim_{n\to
\infty}\|T_{2}-S_{n}T_{1}S_{n}^{-1}\|=0$ and $T_{2}-S_{n}T_{1}S_{n}^{-1}\in\mathcal{K(H)}$.
\end{theorem}
\begin{proof}
Let $(R_{n})_{n\geq 1}$ be a sequence in $\mathcal{B(H)}$ with $\|R_{n}\|\|R_{n}^{-1}\|\leq\lambda$ such that $\displaystyle\lim_{n\to\infty}\|T_{2}-R_{n}T_{1}R_{n}^{-1}\|
=0$. Without loss of generality, we may assume that $\|R_{n}\|\leq\lambda$ and $\|R_{n}^{-1}\|\leq 1$ so that $\displaystyle\sup_{n\geq 1}\|R_{n}\|,\sup_{n\geq 1}\|R_{n}^{
-1}\|<\infty$. Then
\[T_{2}^{\mathscr{U}}=(R_{1}T_{1}R_{1}^{-1},R_{2}T_{1}R_{2}^{-1},R_{3}T_{1}R_{3}^{-1},\ldots)_{\mathscr{U}}=(R_{1},R_{2},R_{3},\ldots)_{\mathscr{U}}T_{1}^{\mathscr{U}}(R_{
1},R_{2},R_{3},\ldots)_{\mathscr{U}}^{-1}.\]
Hence, $T_{1}^{\mathscr{U}}$ is $\lambda$-similar to $T_{2}^{\mathscr{U}}$, and so $T_{1}\oplus\widehat{T}_{1}$ is $\lambda$-similar to $T_{2}\oplus\widehat{T}_{2}$. By a
variation of Lemma \ref{32}, there exists a separable reducing subspace $\mathcal{M}$ for both $\widehat{T}_{1}$ and $\widehat{T}_{2}$ such that $T_{1}\oplus(\widehat{T}_{
1}|_{\mathcal{M}})$
is $\lambda$-similar to $T_{2}\oplus(\widehat{T}_{2}|_{\mathcal{M}})$. By Theorem \ref{34}, the result follows.
\end{proof}
The preceding theorem was proved in \cite{Voiculescu} for $\lambda=1$ (i.e., $T_{2}\in\mathcal{U}(T_{1})^{-\|\,\|}\Rightarrow T_{1}\simeq_{a}T_{2}$) by applying Theorem \ref{33} in a different way.

The rest of this paper is mainly devoted to proving Theorem \ref{37} below.

In the sequel, we say that an operator $T_{1}\in \mathcal{B(H)}$ is a restriction of another operator $T_{2}\in \mathcal{B(H)}$ to mean that $T_{1}$ is unitarily
equivalent to a restriction of $T_{2}$. We do the same thing for compression and reducing part. This is to simplify our presentation.
\begin{theorem}[\cite{Hadwin1}, Theorem 4.3 and \cite{Hadwin2}, Theorem 4.4]\label{37}
Let $T\in \mathcal{B(H)}$. Then
\begin{equation}\label{31e}
\mathcal{U}(T)^{-SOT}=\{B\in \mathcal{B(H)}:B\text{ is a restriction of an operator in }\mathcal{U}(T)^{-\|\,\|}\},
\end{equation}
\begin{equation*}
\mathcal{U}(T)^{-WOT}=\{B\in \mathcal{B(H)}:B\text{ is a compression of an operator in }\mathcal{U}(T)^{-\|\,\|}\}.
\end{equation*}
\begin{equation*}
\mathcal{U}(T)^{-*\hyphen SOT}=\{B\in \mathcal{B(H)}:B\text{ is a reducing part of an operator in }\mathcal{U}(T)^{-\|\,\|}\},
\end{equation*}
\end{theorem}
The idea of this result is the following lemma.
\begin{lemma}\label{38}
Let $(T_{n})_{n\geq 1}$ be a sequence in $\mathcal{B(H)}$ and let $B\in\mathcal{B(H)}$.
\begin{enumerate}[(1)]
\item If $T_{n}\to B$ in SOT then $B$ is a restriction of $(T_{1},T_{2},\ldots)_{\mathscr{U}}$.
\item If $T_{n}\to B$ in WOT then $B$ is a compression of $(T_{1},T_{2},\ldots)_{\mathscr{U}}$.
\item If $T_{n}\to B$ in $*$-SOT then $B$ is a redcing part of $(T_{1},T_{2},\ldots)_{\mathscr{U}}$.
\end{enumerate}
\end{lemma}
\begin{proof}
By the uniform boundedness principle, $\displaystyle\sup_{n\geq 1}\|T_{n}\|<\infty$ so that the ultraproduct\\
$(T_{1},T_{2},\ldots)_{\mathscr{U}}$ is well defined.

If $T_{n}\to B$ in SOT then $\{(x)_{\mathscr{U}}:x\in\mathcal{H}\}$ is invariant under $(T_{1},T_{2},\ldots)_{\mathscr{U}}$, and
$B\cong (T_{1},T_{2},\ldots)_{\mathscr{U}}|_{\{(x)_{\mathscr{U}}:x\in\mathcal{H}\}}$.

Suppose that $T_{n}\to B$ in WOT. Recall that the orthogonal projection from $\mathcal{H}^{\mathscr{U}}$ onto $\{(x)_{\mathscr{U}}:x\in\mathcal{H}\}$ is given by
$\displaystyle (x_{n})_{\mathscr{U}}\mapsto(w\hyphen\lim_{k,\mathscr{U}}x_{k})_{\mathscr{U}}$. Thus, the compression of $(T_{1},T_{2},\ldots)_{\mathscr{U}}$ to
$\{(x)_{\mathscr{U}}:x\in\mathcal{H}\}$ is given by $\displaystyle(x)_{\mathscr{U}}\mapsto (w\hyphen\lim_{k,\mathscr{U}}T_{k}x)_{\mathscr{U}}=(Bx)_{\mathscr{U}}$. Hence,
$B$ is a compression of $(T_{1},T_{2},\ldots)_{\mathscr{U}}$.

If $T_{n}\to B$ in $*$-SOT then $\{(x)_{\mathscr{U}}:x\in\mathcal{H}\}$ is a reducing subspace for $(T_{1},T_{2},\ldots)_{\mathscr{U}}$, and $B\cong (T_{1},T_{2},\ldots)_{
\mathscr{U}}|_{\{(x)_{\mathscr{U}}:x\in\mathcal{H}\}}$.
\end{proof}
\begin{proof}[Proof of Theorem \ref{37}]
If $B\in\mathcal{U}(T)^{-SOT}$ then there exists a sequence $(W_{n})_{n\geq 1}$ of unitary operators in $\mathcal{B(H)}$ such that $W_{n}TW_{n}^{*}\to B$ in SOT.
Thus, by Lemma \ref{38}, $B$ is a restriction of $(W_{1}TW_{1}^{*},W_{2}TW_{2}^{*},\ldots)_{\mathscr{U}}\cong T^{\mathscr{U}}\cong T\oplus\widehat{T}$. Hence, there exists
a separable reducing subspace $\mathcal{M}$ for $\widehat{T}$ such that $B$ is a restriction of $T\oplus(\widehat{T}|_{\mathcal{M}})$. But by Theorem \ref{34},
$T\oplus(\widehat{T}|_{\mathcal{M}})\simeq_{a}T$. Therefore, one inclusion of (\ref{31e}) is proved.

The proof of the other inclusion here is more or less the same as that in \cite{Hadwin2}. But we include it here for self-containedness. To prove this
inclusion, it suffices to show that if $B$ is a restriction of $T$, then $B\in\mathcal{U}(T)^{-SOT}$. This is an immediate consequence of the next lemma. Thus, the proof
of (\ref{31e}) is complete.

The proofs of the other assertions are similar using some variations of the next lemma.
\end{proof}
\begin{lemma}\label{39}
Suppose that $T\in\mathcal{B}(\mathcal{H}\oplus\mathcal{H})$ and that $\mathcal{H}\oplus\{0\}$ is an invariant subspace for $T$. Let $B=T|_{\mathcal{H}\oplus\{0\}}\in
\mathcal{B}(\mathcal{H}\oplus\{0\})$. Then there exists a sequence $(W_{n})_{n\geq 1}$ of unitary operators in $\mathcal{B}(\mathcal{H}\oplus\mathcal{H},\mathcal{H}\oplus
\{0\})$ such that $W_{n}TW_{n}^{*}\to B$ in SOT.
\end{lemma}
\begin{proof}
Let $P_{n}$ be a sequence of finite rank orthogonal projections converging in SOT to the identity operator on $\mathcal{H}$. Let $W_{n}:\mathcal{H}\oplus\mathcal{H}\to
\mathcal{H}\oplus\{0\}$ be a unitary operator such that
\[W_{n}(x,0)=(x,0),\quad x\in P_{n}\mathcal{H}.\]
Then
\[W_{n}[(I-P_{n})\mathcal{H}\oplus\mathcal{H}]=(I-P_{n})\mathcal{H}\oplus\{0\}.\]
For $x\in P_{n}\mathcal{H}$,
\begin{eqnarray*}
(B-W_{n}TW_{n}^{*})(x,0)&=&B(x,0)-W_{n}T(x,0)
\\&=&B(x,0)-W_{n}B(x,0)
\\&=&B(x,0)-W_{n}(P_{n}\oplus 0)B(x,0)\\&&-W_{n}((I-P_{n})\oplus 0)B(x,0)
\\&=&B(x,0)-(P_{n}\oplus 0)B(x,0)\\&&-W_{n}((I-P_{n})\oplus 0)B(x,0)
\\&=&((I-P_{n})\oplus 0)B(x,0)-W_{n}((I-P_{n})\oplus 0)B(x,0),
\end{eqnarray*}
and thus,
\[\|(B-W_{n}TW_{n}^{*})(x,0)\|\leq 2\|((I-P_{n})\oplus 0)B(x,0)\|,\quad x\in P_{n}\mathcal{H}.\]
Hence, for $x\in\mathcal{H}$,
\begin{eqnarray*}
\|(B-W_{n}TW_{n}^{*})(x,0)\|&\leq&\|(B-W_{n}TW_{n}^{*})(P_{n}x,0)\|\\&&+\|(B-W_{n}TW_{n}^{*})((I-P_{n})x,0)\|
\\&\leq&2\|((I-P_{n})\oplus 0)B(P_{n}x,0)\|\\&&+\|B-W_{n}TW_{n}^{*}\|\|(I-P_{n})x\|
\\&\leq&2\|((I-P_{n})\oplus 0)B(x,0)\|\\&&+2\|((I-P_{n})\oplus 0)B((I-P_{n})x,0)\|\\&&+\|B-W_{n}TW_{n}^{*}\|\|(I-P_{n})x\|
\\&\leq&2\|((I-P_{n})\oplus 0)B(x,0)\|\\&&+2\|((I-P_{n})\oplus 0)B((I-P_{n})x,0)\|\\&&+(\|B\|+\|T\|)\|(I-P_{n})x\|
\\&\leq&2\|((I-P_{n})\oplus 0)B(x,0)\|+2\|B\|\|(I-P_{n})x\|\\&&+(\|B\|+\|T\|)\|(I-P_{n})x\|\to 0,
\end{eqnarray*}
as $n\to\infty$. Therefore, $W_{n}TW_{n}^{*}\to B$ in SOT.
\end{proof}
The following result seem to be known. (The results in \cite{Davidson} are somewhat related to this result.)
\begin{theorem}\label{310}
Let $T_{1},T_{2}\in\mathcal{B(H)}$. Suppose that there is a sequence $(P_{n})_{n\geq 1}$ of finite rank orthogonal projections on $\mathcal{H}$ such that $P_{n}\to I$ in
SOT and $P_{n}T_{1}|_{P_{n}\mathcal{H}}$ is a restriction (resp. compression, reducing part) of $T_{2}$. Then $T_{1}$ is a restriction (resp. compression,
reducing part) of an operator in $\mathcal{U}(T_{2})^{-\|\,\|}$.
\end{theorem}
\begin{proof}
The operator $T_{1}$ is a reducing part of $(P_{1}T_{1}|_{P_{1}\mathcal{H}},P_{2}T_{1}|_{P_{2}\mathcal{H}},P_{3}T_{1}|_{P_{3}\mathcal{H}})_{\mathscr{U}}$ via the map
$x\mapsto (P_{n}x)_{\mathscr{U}}$. Hence, by assumption, $T_{1}$ is a restriction of $T_{2}^{\mathscr{U}}\cong T_{2}\oplus\widehat{T}_{2}$. Then we can find a
separable reducing subspace $\mathcal{M}$ for $\widehat{T}_{2}$ such that $T_{1}$ is a restriction of $T_{2}\oplus(\widehat{T}_{2}|_{\mathcal{M}})$. But by Theorem
\ref{33}, $T_{2}\oplus(\widehat{T}_{2}|_{\mathcal{M}})\simeq_{a}T_{2}$. Thus, the result follows.
\end{proof}
We conclude by briefly explaining how the work in this section was dervied. Suppose that the Banach spaces $X_{1},X_{2},\ldots$ have been replaced by operators
$T_{1},T_{2},\ldots\in\mathcal{B(H)}$, respectively. This suggests to replace the ultraproduct $(X_{1},X_{2},\ldots)_{\mathscr{U}}$ by the operator $(T_{1},T_{2},\ldots)_{
\mathscr{U}}$. In other words, the ultraproduct of Banach spaces should be replaced by the ultraproduct of the corresponding operators. The preceding result was motivated by the
concept of finite representability of Banach spaces (see, e.g., \cite[Chapter 8]{Diestel}), which is closely related to ultraproducts of Banach spaces. The other results
Theorem \ref{35} and Theorem \ref{36} and the proof of Theorem \ref{37} were inspired by the proof of the preceding result.

%
{\bf Acknowledgements.} I thank the referee for pointing out that in Proposition \ref{13}, we actually obtain that $A$ is block diagonal. In the original version, I
obtained that $A$ is quasidiagonal. I am grateful to C. Foias and C.~M. Pearcy for their very helpful criticisms and suggestions. I also thank W.~B. Johnson for the proof
of Proposition \ref{13}, which is simpler than the original proof, and other helpful remarks. The author was supported in part by the N.~W.~Naugle Fellowship and the A.~G.
\& M.~E.~Owen Chair in the Department of Mathematics.

\end{document}